\documentclass{article}
\usepackage[english]{babel} 		
\usepackage{fullpage}				
\usepackage[utf8]{inputenc}			
\usepackage{amsmath, amssymb,amsthm}
\usepackage{graphicx}				
\usepackage{hyperref}				
\usepackage{xcolor}
\usepackage{multirow}
\usepackage{tikz}
\usepackage{ytableau}
\usepackage{mathrsfs}
\usepackage{thmtools}
\usepackage{thm-restate}
\usepackage{varwidth}
\usepackage{comment}
\usepackage{amsmath}
\usepackage{mathtools}
\usepackage[mathscr]{euscript}
\usepackage{bbm}
\usepackage{multicol}
\usepackage{enumerate}

\usepackage[
    backend=bibtex,
    maxnames=5,
    style=alphabetic
  ]{biblatex}
\theoremstyle{definition}
\newtheorem{thm}{Theorem}[section]
\newtheorem{lem}[thm]{Lemma}
\newtheorem*{lem*}{Lemma}
\newtheorem*{thm*}{Theorem}
\newtheorem{prop}[thm]{Proposition}

\newtheorem{cor}[thm]{Corollary}

\newtheorem{defn}[thm]{Definition}
\newtheorem*{remark*}{Remark}
\newtheorem{remark}{Remark}
\newtheorem*{example}{Example}

\newtheorem{cor/defn}[thm]{Corollary/Definition}
\newcommand{\y}{\mathscr{Y}}
\DeclareMathOperator{\sH}{\mathscr{H} }

\DeclareMathOperator{\sP}{\mathscr{P}}

\DeclareMathOperator{\GL}{\mathrm{GL}}

\DeclareMathOperator{\Hilb}{\mathrm{Hilb}}
\DeclareMathOperator{\sort}{sort}
\DeclareMathOperator{\Par}{\mathrm{Par}}
\DeclareMathOperator{\Comp}{\mathrm{Comp}}
\DeclareMathOperator{\Compred}{\mathrm{Comp}^{\textit{red}}}
\title{Delta Operators on Almost Symmetric Functions}
\author{Milo Bechtloff Weising}
\date{\today}

\begin{document}

\maketitle

\abstract{We construct $\Delta$-operators $F[\Delta]$ on the space of almost symmetric functions $\sP_{as}^{+}$. These operators extend the usual $\Delta$-operators on the space of symmetric functions $\Lambda \subset \sP_{as}^{+}$ central to Macdonald theory. The $F[\Delta]$ operators are constructed as certain limits of symmetric functions in the Cherednik operators $Y_i$ and act diagonally on the stable-limit non-symmetric Macdonald functions $\widetilde{E}_{(\mu|\lambda)}(x_1,x_2,\ldots;q,t).$ Using properties of Ion-Wu limits, we are able to compute commutation relations for the $\Delta$-operators $F[\Delta]$ and many of the other operators on $\sP_{as}^{+}$ introduced by Ion-Wu. Using these relations we show that there is an action of $\mathbb{B}_{q,t}^{\text{ext}}$ on almost symmetric functions which we show is isomorphic to the polynomial representation of $\mathbb{B}_{q,t}^{\text{ext}}$ constructed by Gonz\'{a}lez-Gorsky-Simental.}

\tableofcontents

\section{Introduction}

The symmetric Macdonald functions $P_{\lambda}[X;q,t]$, introduced by Macdonald \cite{MacDSLC}, are central objects in modern algebraic combinatorics. One of the many characterizations of the symmetric functions $P_{\lambda}[X;q,t]$ comes from considering the action of the Macdonald operator $\Delta$ on the ring of symmetric functions $\Lambda.$ The $P_{\lambda}[X;q,t]$ are a $\Delta$-eigenbasis for the ring of symmetric functions $\Lambda$ with distinct spectrum. 

Through the work of Haiman \cite{HMacDG} it is known that the Macdonald operator and a variant $\widetilde{H}_{\lambda}[X;q,t]$ of the symmetric Macdonald functions $P_{\lambda}$ carry geometric information about the Hilbert schemes $\Hilb^{n}(\mathbb{C}^2)$ of points in $\mathbb{C}^2.$ This work was extended by Schiffmann-Vasserot in \cite{SV} where the authors constructed a representation of the positive elliptic Hall algebra $\mathscr{E}^{+}$ of Burban-Schiffmann \cite{BS} on the ring of symmetric functions $\Lambda$ generated by multiplication operators $p_r[X]^{\bullet}$ and the $\Delta$-operators $p_{\ell}[\Delta]$. These $\Delta$-operators act diagonally on the symmetric Macdonald function basis and satisfy many interesting commutation relations encoded by the positive elliptic Hall algebra $\mathscr{E}^{+}$. One may construct these operators directly as stable limits of certain elements of the spherical double affine Hecke algebras in type $\GL.$ In particularly, the operator $p_1[\Delta] = \Delta$ is essentially the stable limit of operators of the form $Y_{1}^{(n)}+\ldots +Y_{n}^{(n)}$ acting on symmetric polynomials where $Y_{i}^{(n)}$ denote the Cherednik operators. 

Extending the work of Schiffmann-Vasserot, Carlsson-Gorksy-Mellit constructed an algebra $\mathbb{B}_{q,t}$ which acts on the equivariant K theory of the parabolic flag Hilbert schemes of points in $\mathbb{C}^2$. This algebra is a non-unital quiver path algebra containing many copies of affine Hecke algebras in type $\GL.$ A highly related algebra $\mathbb{A}_{q,t}$ introduced by Carlsson-Mellit, called the double Dyck path algebra, was a crucial ingredient in their proof of the Shuffle Theorem \cite{CM_2015}. By looking at loops at the vertex labelled $0$ of $\mathbb{B}_{q,t}$ in the polynomial representation $V_{\bullet}$ one may recover many of the operators coming from the action of $\mathscr{E}^{+}$ on $V_{0} = \Lambda.$ However, one does not find the $\Delta$-operators in this way. More recently, the authors Gonz\'{a}lez-Gorsky-Simental \cite{gonzález2023calibrated} studied a larger algebra $\mathbb{B}_{q,t}^{\text{ext}}$ which roughly adds in these missing $\Delta$-elements to $\mathbb{B}_{q,t}$. They construct many representations for this algebra but in particular they show that there is a unique action of $\mathbb{B}_{q,t}^{\text{ext}}$ on $V_{\bullet}$ which extends the usual action of $\mathbb{B}_{q,t}$ and contains the $\Delta$-operators and their analogues on the equivariant K theory of parabolic flag Hilbert schemes with flags of non-trivial length.

Through the work of Ion-Wu \cite{Ion_2022} we are able to understand $\mathbb{A}_{q,t}$ and $\mathbb{B}_{q,t}$ through the double affine Hecke algebras in type $\GL.$ This relationship comes through the existence of the $+$-stable limit double affine Hecke algebra $\mathscr{H}$ and its action on the ring of almost symmetric functions $\sP_{as}^{+}.$ This action in a sense globalizes the polynomial representations of $\mathbb{A}_{q,t}$ and $\mathbb{B}_{q,t}$ and is constructed from the polynomial representations of the double affine Hecke algebras in type $\GL.$ Importantly, Ion-Wu construct limit Cherednik operators $\y_i$ which correspond to the action of the $z_i$ elements in $\mathbb{A}_{q,t}$ and $\mathbb{B}_{q,t}$. The process used by Ion-Wu for constructing this algebra $\mathscr{H}^{+}$ and its representation on $\sP_{as}^{+}$ is a not the classical stable limit procedure but rather a $t$-adically enriched version of stability. The additional flexibility in the notion of convergence for sequences of polynomials is enough to make sense of a version of Cherednik theory for the  $+$-stable limit double affine Hecke algebra $\mathscr{H}^{+}$ including the existence of a special basis $\widetilde{E}_{(\mu|\lambda)}(x_1,x_2,\ldots;q,t)$ for $\sP_{as}^{+}$ of weight vectors for the limit Cherednik operators $\y$ \cite{MBWArxiv}. However, the spectrum of the limit Cherednik operators $\y$ is far from simple which each nonzero weight space being infinite dimensional. To remedy this issue the author constructed a new operator $\Psi_{p_1}$ on $\sP_{as}^{+}$ extending the usual Macdonald operator on $\Lambda \subset \sP_{as}^{+}$ which commutes with the limit Cherednik operators $\y$ and refines the nonzero $\y$-weight spaces to be $1$ dimensional. This operator is constructed as an Ion-Wu operator limit of the operators $t^nY_{1}^{(n)}+\ldots + t^nY_{n}^{(n)}$ acting on finite variable polynomials. The author conjectured that a similar operator may be constructed for any symmetric function $F$ i.e. that the sequence $F[t^nY_{1}^{(n)}+\ldots + t^nY_{n}^{(n)}]$ converges in the sense of Ion-Wu to an operator on $\sP_{as}^{+}.$

In this paper we confirm the author's conjecture from \cite{MBWArxiv} and construct a family of operators $F[\Delta]$ on the ring of almost symmetric functions $\sP_{as}^{+}$ built as Ion-Wu limits of symmetric polynomials in the Cherednik operators (Theorem \ref{convergence theorem}). These operators act diagonally in the $\widetilde{E}_{(\mu|\lambda)}$ basis and extend the action of the $\Delta$-operators on $\Lambda \subset \sP_{as}^{+}.$ The proof that these sequences of symmetric polynomials in the Cherednik operators converge as operators is highly technical and involves understanding the action of the operators $F[t^nY_{1}^{(n)}+\ldots + t^nY_{n}^{(n)}]$ on partially symmetric polynomials in finitely many variables. Note that this is more involved than understanding the action of symmetric polynomials in the Cherednik operators on symmetric polynomials which is well understood. Importantly, the $F[\Delta]$ cannot be expressed entirely in terms of the limit Cherednik operators $\y.$ We will obtain an explicit expansion (Proposition \ref{convergence for elementary sym}) of the operators $e_r[\Delta]$ acting on each of the subspaces $x_1\cdots x_k \sP_{as}^{+}$, which depends on $k$, in terms of the operators $\y_i,$ $\epsilon_{k},$ and $\pi.$ Using properties of Ion-Wu operator limits we compute some commutation relations (Propositions \ref{commutation relation for Delta}) between the $F[\Delta]$ and the operators $\y_i,T_i,\widetilde{\pi}, $ and $\epsilon_{k}.$ At the end of the paper we will use these commutation relations to show that there exists an action of $\mathbb{B}_{q,t}^{\text{ext}}$ on $\bigoplus_{k \geq 0}x_1\cdots x_k \sP(k)^{+}$ which we show is isomorphic to the polynomial representation $V_{\bullet}$ of $\mathbb{B}_{q,t}^{\text{ext}}$ constructed by Gonz\'{a}lez-Gorsky-Simental (Corollary \ref{ext B qt cor}). This result shows that we may consider the $\Delta$-operators on $V_{\bullet}$ in \cite{gonzález2023calibrated} as limits of symmetric polynomials in the Cherednik operators. 

\subsection{Acknowledgements}

The author would like to thank their advisor Monica Vazirani for her continued support. The author would also like to thank Nicolle Gonz\'{a}lez, Jos\'{e} Simental, and Eugene Gorsky for helpful conversations about $\mathbb{B}_{q,t}^{\text{ext}}$.

\section{Definitions and Conventions}

We present here the conventions that will be used in this paper for the double affine Hecke algebra of type $GL$. Take note of the quadratic relation $(T_{i}-1)(T_{i}+t) = 0$ which has been chosen to match with the conventions in \cite{Ion_2022} and the author's prior work \cite{MBWArxiv}. Some of the other notation has been slightly altered to align with other works of the author.

\begin{defn} \label{defn1}
Define the \textbf{\textit{double affine Hecke algebra}} $\mathscr{H}_n$ to be the $\mathbb{Q}(q,t)$-algebra generated by $T_1,\ldots,T_{n-1}$, $X_1^{\pm 1},\ldots,X_{n}^{\pm 1}$, and $Y_1^{\pm 1},\ldots,Y_n^{\pm 1}$ with the following relations:

\begin{multicols}{2}
\begin{enumerate}[(i)]
    \item 
    \label{def-i} 
    $(T_i -1)(T_i +t) = 0$,
    \item [] $T_iT_{i+1}T_i = T_{i+1}T_iT_{i+1}$,
    \item [] $T_iT_j = T_jT_i$, $|i-j|>1$,
    \item 
    \label{def-ii} $T_i^{-1}X_iT_i^{-1} = t^{-1}X_{i+1}$,
    \item []$T_iX_j = X_jT_i$, $j \notin \{i,i+1\}$,
    \item []$X_iX_j = X_jX_i$,
    \item 
    \label{def-iii}$T_iY_iT_i = tY_{i+1}$,
    \item []$T_iY_j = Y_jT_i$, $j\notin \{i,i+1\}$,
    \item []$Y_iY_j = Y_jY_i$,
    \item 
    \label{def-iv}$Y_1T_1X_1 = X_2Y_1T_1$,
    \item 
    \label{def-v}$Y_1X_1\cdots X_n = qX_1\cdots X_nY_1$
    \item []
\end{enumerate}
\end{multicols}

Further, define the special elements $\pi_n, \widetilde{\pi}_n$ by 
\begin{itemize}
    \item $\pi_n := Y_1T_1\cdots T_{n-1}$
    \item $\widetilde{\pi}_n: = X_1T_1^{-1}\cdots T_{n-1}^{-1}.$
\end{itemize}

This conveniently allows us to write $$Y_1 = \pi_nT_{n-1}^{-1}\cdots T_1^{-1}.$$ When required we will write $Y_i^{(n)}$ for the element $Y_i$ in $\mathscr{H}_n$ to differentiate between the element $Y_i^{(m)}$ in a different $\mathscr{H}_m$ for $n \neq m$.
\end{defn}

The following commutation relations may be verified directly:
\begin{itemize}
    \item $Y_1\widetilde{\pi}_n = \widetilde{\pi}_n qY_n$
    \item $Y_{i+1} \widetilde{\pi}_n = \widetilde{\pi}_n Y_{i}$ for $1 \leq i \leq n-1.$
\end{itemize}

\begin{defn}\label{idempotents}
For all $0 \leq k < n$ define the element $\epsilon^{(n)}_k \in \sH_n$ as 
\begin{equation*}\label{idempotent eq version 1}
\epsilon^{(n)}_k : = \frac{1}{[n-k]_{t}!} \sum_{\sigma \in \mathfrak{S}_{(1^k,n-k)}} t^{{n-k \choose 2} - \ell(\sigma)} T_{\sigma}.
\end{equation*}
Here $\mathfrak{S}_{(1^k,n-k)}$ is the Young subgroup of $\mathfrak{S}_{n}$ corresponding to the composition $(1^k,n-k)$, $T_{\sigma} = T_{s_{i_1}}\cdots T_{s_{i_r}}$ whenever $\sigma = s_{i_1}\cdots s_{i_r}$ is a reduced word representing $\sigma$, and $[m]_{t}! : = \prod_{i=1}^{m}(\frac{1-t^i}{1-t})$ is the $t$-factorial. We will simply write $\epsilon^{(n)}$ for $\epsilon^{(n)}_0$.
\end{defn}

\begin{defn} \label{defn2}
Let $\mathscr{P}_n = \mathbb{Q}(q,t)[x_1^{\pm 1},\ldots,x_n^{\pm 1}]$. The \textbf{\textit{ standard representation}} of $\mathscr{H}_n$ is given by the following action on $\mathscr{P}_n$:

\begin{center}
\begin{itemize}
    \item $T_if(x_1,\ldots,x_n) = s_i f(x_1,\ldots,x_n) +(1-t)x_i \frac{1-s_i}{x_i-x_{i+1}}f(x_1,\ldots,x_n)$
    \item $X_if(x_1,..,x_n)= x_if(x_1,\ldots,x_n)$
    \item $\pi_nf(x_1,\ldots,x_n) = f(x_2,\ldots,x_{n}, qx_1)$
\end{itemize}
\end{center}

Here $s_i$ denotes the operator that swaps the variables $x_i$ and $x_{i+1}$. Under this action the $T_i$ operators are known as the \textbf{\textit{Demazure-Lusztig operators}}. For $q,t$ generic $\mathscr{P}_n$ is known to be a faithful representation of $\mathscr{H}_n$. The action of the elements $Y_1,\ldots,Y_n \in \mathscr{H}_n$ are called \textbf{\textit{Cherednik operators}}.
\end{defn}

Set $\mathscr{H}_n^{+}$ to be the positive part of $\mathscr{H}_n$ i.e. the subalgebra generated by $T_1,\ldots,T_{n-1}$, $X_1,\ldots,X_n$, and $Y_1,\ldots,Y_n$ without allowing for inverses in the $X_i$ and $Y_i$ elements and set $\mathscr{P}_n^{+} = \mathbb{Q}(q,t)[x_1,\ldots,x_n]$. Importantly, 
$\mathscr{P}_n^{+}$ is a $\mathscr{H}_n^{+}$-submodule of $\mathscr{P}_n$.

\subsection{Non-symmetric Macdonald Polynomials}
Before discussing non-symmetric Macdonald polynomials we must first review some basic combinatorial definitions.

\begin{defn}
 In this paper, a \textbf{\textit{composition}} will refer to a finite tuple $\mu = (\mu_1,\ldots,\mu_n)$ of non-negative integers. We allow for the empty composition $\emptyset$ with no parts. We will let $\Comp$ denote the set of all compositions. The length of a composition $\mu = (\mu_1,\ldots,\mu_n)$ is $\ell(\mu) = n$ and the size of the composition is $| \mu | = \mu_1+\ldots+\mu_n$. As a convention we will set $\ell(\emptyset) = 0$ and $|\emptyset| = 0.$ We say that a composition $\mu$ is \textbf{\textit{reduced}} if $\mu = \emptyset$ or $\mu_{\ell(\mu)} \neq 0.$ We will let $\Compred$ denote the set of all reduced compositions. Given two compositions $\mu = (\mu_1,\ldots,\mu_n)$ and $\beta = (\beta_1,\ldots,\beta_m)$, define $\mu * \beta = (\mu_1,\ldots,\mu_n,\beta_1,\ldots,\beta_m)$. A \textbf{\textit{partition}} is a composition $\lambda = (\lambda_1,\ldots,\lambda_n)$ with $\lambda_1\geq \ldots \geq \lambda_n \geq 1$. Note that vacuously we allow for the empty partition $\emptyset.$ We denote the set of all partitions by $\Par$. We denote $\sort(\mu)$ to be the partition obtained by ordering the nonzero elements of $\mu$ in weakly decreasing order. The dominance ordering for partitions is defined by $\lambda \trianglelefteq \nu$ if for all $i\geq 1$, $\lambda_1 + \ldots +\lambda_i \leq \nu_1 + \ldots +\nu_i$ where we set $\lambda_i = 0$ whenever $i > \ell(\lambda)$ and similarly for $\nu$. If $\lambda \trianglelefteq \nu$ and $\lambda \neq \nu.$ we will write $\lambda \triangleleft \nu$. We will let $\Phi$ denote the set $\Compred \times \Par $ of all pairs $(\mu|\lambda)$ with $\mu \in \Compred$ and $\lambda \in \Par.$
 
 We will in a few instances use the notation $\mathbbm{1}(p)$ to denote the value $1$ if the statement p is true and $0$ otherwise. If $\mu$ is any composition with size $n$ we let $\mathfrak{S}_n/\mathfrak{S}_{\mu}$ denote the corresponding set of all minimal length left coset representatives. 
 \end{defn}

In line with the conventions in \cite{haglund2007combinatorial} we define the Bruhat order on the type $GL_n$ weight lattice $\mathbb{Z}^{n}$ as follows. 

\begin{defn}
Let $e_1,...,e_n$ be the standard basis of $\mathbb{Z}^n$ and let $\alpha \in \mathbb{Z}^n$. We define the\textbf{\textit{ Bruhat ordering}} on $\mathbb{Z}^n$, written simply by $<$, by first defining cover relations for the ordering and then taking their transitive closure. If $i<j$ such that $\alpha_i < \alpha_j$ then we say $\alpha > (ij)(\alpha)$ and additionally if $\alpha_j - \alpha_i > 1$ then $(ij)(\alpha) > \alpha + e_i - e_j$ where $(ij)$ denotes the transposition swapping $i$ and $j.$
\end{defn}

\begin{defn} \label{defn3}
The \textbf{\textit{non-symmetric Macdonald polynomials}} (for $GL_n$) are a family of Laurent polynomials $E_{\mu} \in \mathscr{P}_n$ for $\mu \in \mathbb{Z}^n$ uniquely determined by the following:

\begin{itemize}
    \item Triangularity: Each $E_{\mu}$ has a monomial expansion of the form $E_{\mu} = x^{\mu} + \sum_{\lambda < \mu} a_{\lambda}x^{\lambda}$

    \item Weight Vector: Each  $E_{\mu}$ is a weight vector for the operators $Y_1^{(n)},\ldots,Y_n^{(n)} \in \mathscr{H}_n$.
\end{itemize}
\end{defn}

The non-symmetric Macdonald polynomials are a $Y^{(n)}$-weight basis for the $\mathscr{H}_n$ standard representation $\mathscr{P}_n$. For $\mu \in \mathbb{Z}^n$, $E_{\mu}$ is homogeneous with degree $\mu_1+\ldots +\mu_n$. Further, the set of $E_{\mu}$ corresponding to $\mu \in \mathbb{Z}_{\geq 0}^n$ gives a basis for $\mathscr{P}_n ^{+}$.

\subsection{Stable-Limit DAHA of Ion-Wu}
As the index $n$ varies, the standard $\mathscr{H}_n$ representations, $\sP_n$, fail to form a direct/inverse system of compatible $\mathscr{H}_n$ representations. However, as the authors Ion and Wu investigate in \cite{Ion_2022}, this sequence of representations is compatible enough to allow for the construction of a limiting representation for a new algebra resembling a direct limit of the double affine Hecke algebras of type $GL$. We will start by giving the definition of this algebra.

\begin{defn}\cite{Ion_2022} \label{defn4}
The $+$-\textit{\textbf{stable-limit double affine Hecke algebra}} of Ion and Wu, $\mathscr{H}^{+}$, is the algebra  generated over $\mathbb{Q}(q,t)$ by the elements $T_i,X_i,Y_i$ for $i \geq 1$ satisfying the following relations:

\begin{center}
\begin{itemize}
    \item The generators $T_i,X_i$ for $i \geq 1$ satisfy 
    \eqref{def-i}  and \eqref{def-ii}
    of Defn. \ref{defn1}.
    \item The generators $T_i,Y_i$ for $i \geq 1$ satisfy 
    \eqref{def-i} and \eqref{def-iii}
    of Defn. \ref{defn1}.
    \item  $Y_1T_1X_1 = X_2Y_1T_1.$
 \end{itemize}
 \end{center}
 \end{defn}

 Note that $X_i,Y_i$ are not invertible in $\sH^{+}.$

\begin{defn}
    Define the \textit{\textbf{ring of symmetric functions}} $\Lambda$ as the inverse limit of the graded rings $\mathbb{Q}(q,t)[x_1,\ldots,x_n]^{\mathfrak{S}_n}$ with respect to the homogeneous quotient maps $f(x_1,\ldots,x_{n+1}) \rightarrow f(x_1,\ldots, x_n,0)$ in the category of graded rings. We will write $X := x_1+x_2+\ldots$ and use plethsytic notation $F(x_1,x_2,\ldots) = F[X].$ Define for $ \ell \geq 1$ the \textit{\textbf{power sum symmetric function}}  $p_{\ell}[X]$ by 
    $$p_{\ell}[X]:= \sum_{i \geq 1} x_i^{\ell}.$$ For $r \geq 1$ define the \textit{\textbf{elementary symmetric function}} $e_r[X]$ as 
    $$e_r[X]:= \sum_{i_1 < \ldots < i_r}x_{i_1}\cdots x_{i_r}.$$ As a convention we set $e_0[X]:= 1.$ If $\lambda = (\lambda_1,\ldots, \lambda_r)$ is a partition define $e_{\lambda}[X]:= e_{\lambda_1}[X]\cdots e_{\lambda_r}[X].$ For a partition $\lambda$ define the \textbf{\textit{monomial symmetric function}} $m_\lambda[X]$ by 
    $$m_\lambda[X] := \sum_{\mu} x^{\mu}$$
    where we range over all distinct monomials $x^{\mu}$ such that $\sigma(\mu) = \lambda$ for some permutation $\sigma$.
\end{defn}

\begin{remark}
    It is a classical result that the sets $\{e_{\lambda}[X]\}_{\lambda \in \Par}$ and $\{m_{\lambda}[X] \}_{\lambda \in \Par}$ are $\mathbb{Q}(q,t)$-bases for $\Lambda.$ In fact, 
    $$\Lambda \cong \mathbb{Q}(q,t)[e_1,e_2,\ldots].$$
\end{remark}

\begin{defn}\cite{Ion_2022} \label{defn5}
 Let $\mathscr{P}_{\infty}^{+}$ denote the inverse limit of the rings $\mathscr{P}_{k}^{+}$ with respect to the homomorphisms $\Xi_k: \sP_{k+1}^{+} \rightarrow \sP_k^{+}$ which send $x_{k+1}$ to 0 at each step. We can naturally extend $\Xi_k$ to a map $\mathscr{P}_{\infty}^{+} \rightarrow \mathscr{P}_k$ which will be given the same name. Let $\mathscr{P}(k)^{+} := \mathbb{Q}(q,t)[x_1,\ldots,x_k]\otimes \Lambda[x_{k+1}+x_{k+2}+\ldots] \subset \mathscr{P}_{\infty}^{+}$. Define the \textbf{\textit{ring of almost symmetric functions}} by $\mathscr{P}_{as}^{+} := \bigcup_{k\geq 0} \mathscr{P}(k)^{+}$. Note $\mathscr{P}_{as}^{+} \subset \mathscr{P}_{\infty}^{+}.$ Define $\rho: \mathscr{P}_{as}^{+} \rightarrow x_1\mathscr{P}_{as}^{+}$ to be the linear map defined by $\rho(x_1^{a_1}\cdots x_n^{a_n}F[x_{m}+x_{m+1}+\ldots]) = \mathbbm{1}(a_1 > 0) x_1^{a_1}\cdots x_n^{a_n}F[x_{m}+x_{m+1}+\ldots] $ for $F \in \Lambda$. Note that $\rho$ restricts to maps $\sP_n \rightarrow x_1\sP_n$ which are compatible with the quotient maps $\Xi_n$.
 \end{defn}

The ring $\sP_{as}^{+}$ is a free graded $\Lambda$-module with homogeneous basis given simply by the set of monomials $x^{\mu}$ with $\mu$ reduced. Therefore, $\sP_{as}^{+}$ has a homogeneous $\mathbb{Q}(q,t)$ basis given by all $x^{\mu}m_{\lambda}[X]$ ranging over all reduced compositions $\mu$ and partitions $\lambda$. Further, the dimension of the homogeneous degree d part of $\mathscr{P}(k)^{+}$ is equal to the number of pairs $(\mu|\lambda)$ of reduced compositions $\mu$ and partitions $\lambda$ with $|\mu|+|\lambda| = d$ and $\ell(\mu) \leq k$. 
 
In order to define the operators required for Ion and Wu's main construction we must first review the new definition of convergence introduced in \cite{Ion_2022}. 

 \begin{defn} \cite{Ion_2022}\label{defn6}
 Let $(f_m)_{m \geq 1}$ be a sequence of polynomials with $f_m \in \mathscr{P}_m^{+}$. Then the sequence $(f_m)_{m \geq 1}$ is \textbf{\textit{convergent}} if there exist some N and auxiliary sequences $(h_m)_{m\geq1}$, $(g^{(i)}_m)_{m\geq 1}$, and $(a^{(i)}_m)_{m\geq 1}$ for $1\leq i \leq N$ with $h_m, g^{(i)}_m \in \mathscr{P}_{m}^{+}$, $a^{(i)}_m \in \mathbb{Q}(q,t)$ with the following properties:

\begin{itemize}
    \item For all $m$, $f_m = h_m + \sum_{i=1}^{N} a^{(i)}_m g^{(i)}_m$.
    \item The sequences $(h_m)_{m\geq1}$, $(g^{(i)}_m)_{m\geq1}$ for $1\leq i \leq N$ converge in $\mathscr{P}_{\infty}^{+}$ with limits $h,g^{(i)}$ respectively. That is to say, $\Xi_m(h_{m+1}) = h_m$ and $\Xi_m(g_{m+1}^{(i)}) = g_{m}^{(i)}$ for all $1\leq i \leq N$ and $m \geq 1$. Further, we require $g^{(i)} \in \mathscr{P}_{as}^{+}$.
    \item The sequences $a^{(i)}_m$ for $1\leq i \leq N$ converge with respect to the t-adic topology on $\mathbb{Q}(q,t)$ with limits $a^{(i)}$ which are required to be in $\mathbb{Q}(q,t)$.
\end{itemize}

The sequence is said to have a limit given by
$\lim_{m} f_m = h + \sum_{i=1}^{N}a^{(i)}g^{(i)}.$

\end{defn}

\begin{remark}\label{useful defn of convergence}
    In this paper we will be entirely concerned with convergent sequences $(f_m)_{m \geq 1}$ with almost symmetric limits $\lim_{m} f_m \in \sP_{as}^{+}$. In this case it follows readily from definition that each of these convergent sequences necessarily will have the form 
    $$f_m(x_1,\ldots,x_m) = \sum_{i=1}^{N} c_i^{(m)}x^{\mu^{(i)}}F_i[x_1+\ldots + x_m]$$
    where $N \geq 1$ is fixed, $c_i^{(m)}$ are convergent sequences of scalars with $\lim_{m} c_i^{(m)} \in \mathbb{Q}(q,t)$, $F_i$ are symmetric functions, and $\mu^{(i)}$ are compositions. Here we will consider $x^{\mu^{(i)}} = 0$ in $\sP_{m}$ whenever $\ell(\mu^{(i)}) > m$.
\end{remark}

\begin{defn} \cite{Ion_2022}
    For $m \geq 1$ suppose $A_m$ is an operator on $\sP_m^{+}$. The sequence $(A_m)_{m \geq 1}$ of operators is said to \textbf{\textit{converge}} if for every $f \in \sP_{as}^{+}$ the sequence $(A_m(\Xi_m(f)))_{m \geq 1}$ converges to an element of $\sP_{as}^{+}$. From \cite{Ion_2022} the corresponding operator on $\sP_{as}^{+}$ given by $A(f): = \lim_{m} A_m(\Xi_m(f))$ is well defined and said to be the limit of the sequence $(A_m)_{m \geq 1}$. In this case we will simply write $A = \lim_{m} A_m$.
\end{defn}

The following important technical proposition of Ion and Wu will be used repeatedly in this paper. 

\begin{prop}\label{limits proposition}[Prop. 6.21 \cite{Ion_2022}] 
    If $A = \lim_{m} A_m$ and $f = \lim_{m} f_m$ are limit operators and limit functions respectively then $A(f) = \lim_{m} A_m(f_m).$
\end{prop}

This is a sort of continuity statement for convergent sequences of operators. The utility of the above proposition is that for an operator arising as the limit of finite variable operators, $A = \lim_{m} A_m$ say, we can use \textbf{\textit{any}} sequence $(f_m)_{m \geq 1}$ converging to $f \in \sP_{as}^{+}$ in order to calculate $A(f)$. 

It is easy to verify the following proposition using Proposition \ref{limits proposition}.

\begin{prop}\label{product of limits proposition}\cite{Ion_2022}
    If $A = \lim_{m} A_m$ and $B = \lim_{m} B_m$ then $AB = \lim_{m} A_mB_m.$
\end{prop}

\subsection{Important Limit Operators}

In this section we detail the construction of some of the important operators on $\sP_{as}^{+}$ which will be required later. The fact that the operators as presented below are in fact well defined is non-trivial and follows from \cite{Ion_2022} and \cite{MBWArxiv}.

\begin{defn}
We define the operators $X_i$,$T_i$,$\y_i$,$\pi$, $\widetilde{\pi}$ on $\sP_{as}^{+}$ by the following limits:
    \begin{itemize}
        \item $X_i := \lim_{n} X_i $
        \item $T_i := \lim_{n} T_i $
        \item $\epsilon_k := \lim_{n} \epsilon_k^{(n)}$
        \item $\y_i := \lim_{n} t^{n-i+1}T_{i-1}\cdots T_1 \rho \pi_n T_{n-1}^{-1}\cdots T_i^{-1}$
        \item $\pi := \lim_{n} \pi_n$
        \item $\widetilde{\pi} := \lim_{n} \widetilde{\pi}_n.$
    \end{itemize}
\end{defn}

We may explicitly describe the action of some of these operators:

\begin{itemize}
    \item $X_i(f) = x_if$
    \item $T_i(f) = s_if + (1-t)x_i \frac{f-s_if}{x_i-x_{i+1}}$
    \item $\pi(x_1^{a_1}\cdots x_k^{a_k}F[X]) = x_2^{a_1}\cdots x_{k+1}^{a_k}F[X+(q-1)x_1]$
    \item $\widetilde{\pi}(x_1^{a_1}\cdots x_k^{a_k}F[X]) = x_1T_1^{-1}\cdots T_{k}^{-1}x_1^{a_1}\cdots x_k^{a_k}F[X].$ 
\end{itemize}

The operators $\y_i$ were originally constructed by Ion-Wu and are known as the \textbf{\textit{limit Cherednik operators}}. They are limits of the \textit{\textbf{deformed Cherednik operators}} $\widetilde{Y}_i^{(n)}:= t^{n-i+1}T_{i-1}\cdots T_1 \rho \pi_n T_{n-1}^{-1}\cdots T_i^{-1}$ which have the important property that $$\widetilde{Y}_i^{(n)}X_i = t^nY_i^{(n)}X_i.$$ It is non-trivial that not only are the $\y_i$ operators well defined, but that they in fact mutually commute i.e. $[\y_i,\y_j] = 0.$

It is a result of Ion-Wu that the following representation of $\sH^{+}$ is well defined.

\begin{defn}\cite{Ion_2022}
    The standard representation of $\sH^{+}$ consists of $\sP_{as}^{+}$ with the action generated by the operators $X_i, \y_i,T_i.$
\end{defn}

We may give an explicit description of the action of the limit Cherednik operators $\y_i$ using the stable-limit non-symmetric Macdonald functions.

\begin{defn} \label{stable-limit non-sym MacD function defn}\cite{MBWArxiv}
    For $(\mu|\lambda) \in \Phi$ define the stable-limit non-symmetric Macdonald function $\widetilde{E}_{(\mu|\lambda)}$ as 
    $$\widetilde{E}_{(\mu|\lambda)}[x_1,x_2,\ldots;q,t] := \lim_{n} \epsilon_{\ell(\mu)}^{(n)}\left(E_{\mu*\lambda*0^{n-(\ell(\mu)+\ell(\lambda))}}(x_1,\ldots,x_n;q,t)\right). $$
\end{defn}

The following gives an explicit characterization of the action of the $\y_i$ on $\sP_{as}^{+}.$

\begin{thm}\label{weight basis theorem}\cite{MBWArxiv}
    The $\widetilde{E}_{(\mu|\lambda)}$ are a $\y$-weight basis for $\mathscr{P}_{as}^{+}$. Further, we have that 
    $$\y_i(\widetilde{E}_{(\mu|\lambda)}) = \widetilde{\alpha}_{(\mu|\lambda)}(i) \widetilde{E}_{(\mu|\lambda)} $$ where $\widetilde{\alpha}_{(\mu|\lambda)}(i)$ is given explicitly by 

    \[ \widetilde{\alpha}_{(\mu|\lambda)}(i) = 
    \begin{cases}
    q^{\mu_i}t^{\ell(\mu)+\ell(\lambda)+1-\beta_{\mu*\lambda}(i)} & i\leq \ell(\mu) , \mu_i \neq 0 \\
    0 & \text{otherwise}
     \end{cases}
\]

and 
$$\beta_{\nu}(i) := \#\{j: 1\leq j \leq i ~, \nu_j \leq \nu_i\} + \#\{j: i < j \leq \ell(\nu) ~, \nu_i > \nu_j\}.$$
\end{thm}

\begin{example}
    We can see directly that $\y_1(\widetilde{E}_{(1|\emptyset)}) = qt \widetilde{E}_{(1|\emptyset)}:$
    \begin{align*}
        \y_1(\widetilde{E}_{(1|\emptyset)}) &= \y_1(x_1)\\
        &= \lim_{n} t^n \rho \pi_n T_{n-1}^{-1}\cdots T_1^{-1}(x_1)\\
        &= \lim_{n} t^n \rho \pi_n(t^{-(n-1)}x_n) \\
        &= \lim_{n} qt \rho(x_1)\\
        &= \lim_{n} qtx_1\\
        &= qtx_1.\\
    \end{align*}
\end{example}

\section{Constructing Delta Operators on $\sP_{as}^{+}$}

We begin by recalling the definition of the operator $\Psi_{p_1}: \sP_{as}^{+} \rightarrow \sP_{as}^{+}$ from the author's prior work. We opt to rename this operator $\Delta:= \Psi_{p_1}$ to better align with the notation used in Macdonald theory.

\begin{defn}
    For a symmetric function $F$ define the operators $\Psi_{F}^{(n)}: \sP_n^{+} \rightarrow \sP_n^{+}$ by 
    $$\Psi_{F}^{(n)}:= F[t^nY_1^{(n)}+\ldots+ t^n Y_n^{(n)}].$$ For $\nu \in \Par$ we denote 
    $$\kappa_{\nu}(q,t):= \sum_{i=1}^{\infty} q^{\nu_i}t^{i} \in \mathbb{Q}(q,t).$$
\end{defn}

\begin{thm}\label{almost sym macdonald operator}\cite{MBWArxiv}
    The sequence of operators $(\Psi_{p_1}^{(n)})_{n \geq 1}$ converges to an operator $\Delta$ on $\sP_{as}^{+}$ with the following properties:
    \begin{itemize}
        \item $\Delta(\widetilde{E}_{(\mu|\lambda)}) = \kappa_{\sort(\mu*\lambda)}(q,t)\widetilde{E}_{(\mu|\lambda)}$
        \item $[\Delta, \y_i ] = 0$
        \item $[\Delta, T_i] = 0$
        \item The joint $\Delta,\y_1,\y_2,\ldots$ weight spaces in $\sP_{as}^{+}$ are all 1-dimensional.
        \item $$\Delta|_{x_1\cdots x_k\sP(k)^{+}} = \y_1+\ldots + \y_k + \frac{t}{1-t} \epsilon_kT_k \cdots T_1 \pi.$$
    \end{itemize}
\end{thm}
\begin{remark}
    The last property of the operator $\Delta$ listed above is not explicitly stated in \cite{MBWArxiv} but follows from the proof of Theorem \ref{almost sym macdonald operator} and the fact that for $1\leq i \leq k-1$

    \begin{align*}
        &(\y_1 +\ldots + \y_k + \frac{t}{1-t} \epsilon_k T_k \cdots T_1 \pi)T_i \\
        &= (\y_1 +\ldots + \y_k)T_i + \frac{t}{1-t} \epsilon_k T_k \cdots T_1 \pi T_i \\
        &= T_i (\y_1 +\ldots + \y_k) + \frac{t}{1-t} \epsilon_k T_k \cdots T_1 T_{i+1} \pi \\
        &= T_i (\y_1 +\ldots + \y_k) + \frac{t}{1-t} \epsilon_k T_i T_k \cdots T_1 \pi \\
        &= T_i(\y_1 +\ldots + \y_k) + T_i \frac{t}{1-t} \epsilon_k T_k \cdots T_1 \pi \\
        &= T_i (\y_1 +\ldots + \y_k + \frac{t}{1-t} \epsilon_k T_k \cdots T_1 \pi). \\
    \end{align*}
    This fact will be generalized in Theorem \ref{convergence theorem}.
\end{remark}

\begin{remark}
    The operator $\Delta$ cannot be expressed entirely in terms of the limit Cherednik operators $\y_i$ and in particular, 
    $$\Delta \neq \y_1+\y_2+\y_3+\ldots.$$
    It is easy to see this since each of the operators $\y_i$ annihilates $\Lambda$ whereas $\Delta$ acts diagonally on $\Lambda$ with non-zero eigenvalues. This is very different than the finite variable situation where the action of the finite variable Macdonald operator is roughly given by $Y_1^{(n)}+\ldots + Y_n^{(n)}.$
\end{remark}

At the end of the author's prior paper \cite{MBWArxiv} it is conjectured that for any symmetric function $F \in \Lambda$ the sequence of operators $(\Psi_{F}^{(n)})_{n \geq 1}$ converges to an operator on $\sP_{as}^{+}.$ An affirmation of this conjecture has direct implications related to the conjectural partially symmetric elliptic Hall algebras mentioned by Carlsson-Mellit in \cite{CM_2015} and the extended double Dyck path algebra $\mathbb{B}_{q,t}^{\text{ext}}$ of Gonz\'{a}lez-Gorsky-Simental in \cite{gonzález2023calibrated}. 

The main purpose of this section is to give a proof of this conjecture. The proof involves a detailed computation which will be done in stages. We will start with some of the required preliminaries. 

\subsection{Preliminaries}

There are a few elementary technical results we will need before we are able to prove the main result of this section Theorem \ref{convergence theorem}.

For the remainder of this section we consider $n,k,r$ with $ n > k+r \geq 1.$ We begin by expressing certain partially symmetric polynomials in the Cherednik elements $Y_i$ in terms of products of consecutive Cherednik elements.

\begin{lem}\label{expansion of partially elementary symmetric Cherednik operators lemma}
    $$e_r\left[t^nY_{k+1}^{(n)}+\ldots + t^nY_n^{(n)}\right] = \sum_{\sigma \in \mathfrak{S}_{(1^k,n-k)}/\mathfrak{S}_{(1^k,r,n-k-r)}}t^{-\ell(\sigma)}T_{\sigma} t^{rn}Y_{k+1}^{(n)}\cdots  Y_{k+r}^{(n)} T_{\sigma^{-1}}$$
\end{lem}
\begin{proof}
    Notice that for $\sigma \in \mathfrak{S}_{(1^k,n-k)}/\mathfrak{S}_{(1^k,r,n-k-r)}$ the values $\sigma(k+1),\ldots, \sigma(k+r)$ satisfy $k+1 \leq \sigma(k+1)<\ldots < \sigma(k+r) \leq n$ and uniquely determine $\sigma.$ As such there is a natural bijection $\mathfrak{S}_{(1^k,n-k)}/\mathfrak{S}_{(1^k,r,n-k-r)} \rightarrow \{(i_1,\ldots , i_r)| k+1 \leq i_1<\ldots < i_r \leq n-k  \} $ given by $\sigma \rightarrow (\sigma(k+1),\ldots , \sigma(k+r)).$ Hence, it suffices to show that for all such $\sigma$
    $$Y_{\sigma(k+1)}^{(n)}\cdots Y_{\sigma(k+r)}^{(n)} = t^{-\ell(\sigma)}T_{\sigma} Y_{k+1}^{(n)}\cdots  Y_{k+r}^{(n)} T_{\sigma^{-1}}.$$ We proceed by induction on the Bruhat order on $\mathfrak{S}_{(1^k,n-k)}/\mathfrak{S}_{(1^k,r,n-k-r)}$. Clearly, this formula holds for $\sigma = 1.$ 

    Suppose $ k+1 = i_0 \leq i_1 < \ldots < i_r \leq i_{r+1} = n$ with $i_{j+1} - i_j > 1$ for some $ 0 \leq j \leq r.$ Then 
    \begin{align*}
        &Y_{i_1}^{(n)}\cdots Y_{i_{j-1}}^{(n)}Y_{(i_{j})+1}^{(n)}Y_{i_{j+1}}^{(n)}Y_{i_{j+2}}^{(n)} \cdots Y_{n}^{(n)} \\
        &= Y_{i_1}^{(n)}\cdots Y_{i_{j-1}}^{(n)}(t^{-1}T_{i_{j}}Y_{i_{j}}^{(n)}T_{i_{j}})Y_{i_{j+1}}^{(n)}Y_{i_{j+2}}^{(n)} \cdots Y_{n}^{(n)} \\
        &= t^{-1}T_{i_{j}}Y_{i_1}^{(n)}\cdots Y_{i_{r}}^{(n)}T_{i_{j}}.\\
    \end{align*}

    Now if $\sigma, \sigma' \in \mathfrak{S}_{(1^k,n-k)}/\mathfrak{S}_{(1^k,r,n-k-r)}$ are the unique elements with 
    $\sigma(k+\ell)= i_{\ell}$ for $1 \leq \ell \leq r$ and $\sigma' = s_{i_{j}}\sigma.$ Suppose that $Y_{\sigma(k+1)}^{(n)}\cdots Y_{\sigma(k+r)}^{(n)} = t^{-\ell(\sigma)}T_{\sigma} Y_{k+1}^{(n)}\cdots  Y_{k+r}^{(n)} T_{\sigma^{-1}}.$ Then from the above we find that 
    \begin{align*}
        &Y_{\sigma'(k+1)}^{(n)}\cdots Y_{\sigma'(k+r)}^{(n)} \\
        &= Y_{i_1}^{(n)}\cdots Y_{i_{j-1}}^{(n)}Y_{i_{j}+1}^{(n)}Y_{i_{j+1}}^{(n)}Y_{i_{j+2}}^{(n)} \cdots Y_{n}^{(n)}\\
        &= t^{-1}T_{i_{j}}Y_{i_1}^{(n)}\cdots Y_{i_{r}}^{(n)}T_{i_{j}} \\
        &= t^{-1}T_{i_{j}}Y_{\sigma(k+1)}^{(n)}\cdots Y_{\sigma(k+r)}^{(n)} T_{i_{j}}\\
        &=  t^{-(1+\ell(\sigma))}T_{i_{j}} T_{\sigma} Y_{k+1}^{(n)}\cdots  Y_{k+r}^{(n)} T_{\sigma^{-1}} T_{i_{j}}\\
        &= t^{-\ell(\sigma')}T_{\sigma'} Y_{k+1}^{(n)}\cdots  Y_{k+r}^{(n)} T_{(\sigma')^{-1}}.\\
    \end{align*}
\end{proof}

Now we may write a product of consecutive Cherednik elements in terms of $\pi_n.$

\begin{lem}\label{expansion of product of Cherednik operators}
    $$t^{rn}Y_{k+1}^{(n)} \cdots  Y_{k+r}^{(n)} = t^{(n-k)+\ldots + (n-k-r+1)} (T_k\cdots T_1)(T_{k+1}\cdots T_2)\cdots (T_{k+r-1}\cdots T_r) \pi_n^{r} (T_{n-r}^{-1}\cdots T_{k+1}^{-1})\cdots (T_{n-1}^{-1}\cdots T_{k+r}^{-1}) $$
\end{lem}

\begin{proof}
We will show this result by induction. For $r=1$ we see that 
$$t^nY_{k+1}^{(n)} = t^{n-k}T_{k}\cdots T_1 \pi_n T_{n-1}^{-1}\cdots T_{k+1}^{-1}.$$ Now we find 
    \begin{align*}
        &t^{(r+1)n}Y_{k+1}^{(n)}\cdots Y_{k+r+1}^{(n)} \\
        &= t^{rn}Y_{k+1}^{(n)} \cdots  Y_{k+r}^{(n)} t^nY_{k+r+1}^{(n)}\\
        &= t^{(n-k)+\ldots + (n-k-r+1)} (T_k\cdots T_1)\cdots (T_{k+r-1}\cdots T_r) \pi_n^{r} (T_{n-r}^{-1}\cdots T_{k+1}^{-1})\cdots (T_{n-1}^{-1}\cdots T_{k+r}^{-1})t^nY_{k+r+1}^{(n)}\\
        &= t^{(n-k)+\ldots + (n-k-r+1)} (T_k\cdots T_1)\cdots (T_{k+r-1}\cdots T_r) \pi_n^{r} (T_{n-r}^{-1}\cdots T_{k+1}^{-1})\cdots (T_{n-1}^{-1}\cdots T_{k+r}^{-1})\left(t^{n-k-r}T_{k+r}\cdots T_1 \pi_n T_{n-1}^{-1}\cdots T_{k+r+1}^{-1} \right)\\
        &= t^{(n-k)+\ldots+ (n-k-r)} (T_k\cdots T_1)\cdots (T_{k+r-1}\cdots T_r) \pi_n^{r} (T_{n-r}^{-1}\cdots T_{k+1}^{-1})\cdots (T_{n-1}^{-1}\cdots T_{k+r}^{-1})T_{k+r}\cdots T_1 \pi_n T_{n-1}^{-1}\cdots T_{k+r+1}^{-1}.\\
    \end{align*}
    Looking closer we see 
    \begin{align*}
        &(T_{n-r}^{-1}\cdots T_{k+1}^{-1})\cdots (T_{n-1}^{-1}\cdots T_{k+r}^{-1})T_{k+r}\cdots T_1 \\
        &= (T_{n-r}^{-1}\cdots T_{k+1}^{-1})\cdots (T_{n-1}^{-1}\cdots T_{k+r+1}^{-1})T_{k+r-1}\cdots T_1 \\
        &= (T_{n-r}^{-1}\cdots T_{k+1}^{-1})\cdots (T_{n-2}^{-1}\cdots T_{k+r-1}^{-1})T_{k+r-1}\cdots T_1 (T_{n-1}^{-1}\cdots  T_{k+r+1}^{-1}) \\
        &= \ldots \\
        &= (T_k\cdots T_1)(T_{n-r}^{-1}\cdots T_{k+2}^{-1})\cdots (T_{n-1}^{-1}\cdots T_{k+r+1}^{-1}).\\
    \end{align*}
Therefore, 
    \begin{align*}
        & (T_k\cdots T_1)\cdots (T_{k+r-1}\cdots T_r) \pi_n^{r} (T_{n-r}^{-1}\cdots T_{k+1}^{-1})\cdots (T_{n-1}^{-1}\cdots T_{k+r}^{-1})T_{k+r}\cdots T_1 \pi_n T_{n-1}^{-1}\cdots T_{k+r+1}^{-1} \\
        &= (T_k\cdots T_1)\cdots (T_{k+r-1}\cdots T_r) \pi_n^{r} (T_k\cdots T_1)(T_{n-r}^{-1}\cdots T_{k+2}^{-1})\cdots (T_{n-1}^{-1}\cdots T_{k+r+1}^{-1}) \pi_n T_{n-1}^{-1}\cdots T_{k+r+1}^{-1}\\
        &= (T_k\cdots T_1)\cdots (T_{k+r-1}\cdots T_r)(T_{k+r}\cdots T_{r+1}) \pi_n^{r} (T_{n-r}^{-1}\cdots T_{k+2}^{-1})\cdots (T_{n-1}^{-1}\cdots T_{k+r+1}^{-1}) \pi_n T_{n-1}^{-1}\cdots T_{k+r+1}^{-1} \\
        &= (T_k\cdots T_1)\cdots (T_{k+r-1}\cdots T_r)(T_{k+r}\cdots T_{r+1}) \pi_n^{r+1} (T_{n-r-1}^{-1}\cdots T_{k+1}^{-1})\cdots (T_{n-2}^{-1}\cdots T_{k+r}^{-1}) (T_{n-1}^{-1}\cdots T_{k+r+1}^{-1}) \\
    \end{align*}
    so that 
    \begin{align*}
        &t^{(r+1)n}Y_{k+1}^{(n)}\cdots Y_{k+r+1}^{(n)} \\
        &= t^{(n-k)+\ldots+ (n-k-r)}(T_k\cdots T_1)\cdots (T_{k+r}\cdots T_{r+1}) \pi_n^{r+1} (T_{n-r-1}^{-1}\cdots T_{k+1}^{-1})\cdots (T_{n-1}^{-1}\cdots T_{k+r+1}^{-1}).\\
    \end{align*}
\end{proof}

We need the following standard result. 

\begin{lem}\label{t-series computation}
    $$\sum_{\sigma \in \mathfrak{S}_{(1^k,n-k)}/\mathfrak{S}_{(1^k,r,n-k-r)}} t^{{n-k \choose 2} - {n-k-r \choose 2} - {r \choose 2}-\ell(\sigma)} = \frac{[n-k]_{t}!}{[n-k-r]_{t}![r]_t!}$$
\end{lem}

\begin{proof}
    We see the following:
    \begin{align*}
        & [n-k]_{t}! \\
        &= \sum_{\sigma \in \mathfrak{S}_{(1^k,n-k)}} t^{{n-k \choose 2}- \ell(\sigma)}\\
        &= \sum_{\sigma \in \mathfrak{S}_{(1^k,n-k)}/\mathfrak{S}_{(1^k,r,n-k-r)}} \sum_{\gamma \in \mathfrak{S}_{(1^k,r,n-k-r)}} t^{{n-k \choose 2}- \ell(\sigma \gamma)}\\
        &= \sum_{\sigma \in \mathfrak{S}_{(1^k,n-k)}/\mathfrak{S}_{(1^k,r,n-k-r)}} \sum_{\gamma \in \mathfrak{S}_{(1^k,r,n-k-r)}} t^{{n-k \choose 2}- \ell(\sigma) - \ell( \gamma)}\\
        &= \sum_{\sigma \in \mathfrak{S}_{(1^k,n-k)}/\mathfrak{S}_{(1^k,r,n-k-r)}} t^{{n-k \choose 2}- {n-k-r \choose 2} - {r \choose 2}- \ell(\sigma)}  \sum_{\gamma \in \mathfrak{S}_{(1^k,r,n-k-r)}} t^{{n-k-r \choose 2} + {r \choose 2} - \ell( \gamma)}\\
        &= [n-k-r]_{t}![r]_{t}! \sum_{\sigma \in \mathfrak{S}_{(1^k,n-k)}/\mathfrak{S}_{(1^k,r,n-k-r)}} t^{{n-k \choose 2}- {n-k-r \choose 2} - {r \choose 2}- \ell(\sigma)}.\\
    \end{align*}
\end{proof}

Using the prior lemmas in this section now shows the following:

\begin{lem}\label{important lemma}
$$e_r\left[t^nY_{k+1}^{(n)}+\ldots + t^nY_n^{(n)}\right] \epsilon_k^{(n)} = t^{{r +1 \choose 2}} \left(\frac{1-t^{n-k-r+1}}{1-t} \right) \cdots \left(\frac{1-t^{n-k}}{1-t^r} \right) \epsilon_k^{(n)}(T_k\cdots T_1)(T_{k+1}\cdots T_2)\cdots (T_{k+r-1}\cdots T_r) \pi_n^r \epsilon_k^{(n)}.$$
\end{lem}
\begin{proof}
Putting together Lemmas \ref{expansion of partially elementary symmetric Cherednik operators lemma}, \ref{expansion of product of Cherednik operators}, and \ref{t-series computation} we get the following computation:
\begin{align*}
    &e_r\left[t^nY_{k+1}^{(n)}+\ldots + t^nY_n^{(n)}\right] \epsilon_k^{(n)}\\ 
    &= e_r\left[t^nY_{k+1}^{(n)}+\ldots + t^nY_n^{(n)}\right] (\epsilon_k^{(n)})^2\\
    &= \epsilon_k^{(n)}e_r\left[t^nY_{k+1}^{(n)}+\ldots + t^nY_n^{(n)}\right] \epsilon_k^{(n)}\\
    &= \sum_{\sigma \in \mathfrak{S}_{(1^k,n-k)}/\mathfrak{S}_{(1^k,r,n-k-r)}} t^{-\ell(\sigma)} \epsilon_k^{(n)} T_{\sigma} t^{rn}Y_{k+1}^{(n)} \cdots  Y_{k+r}^{(n)} T_{\sigma^{-1}} \epsilon_k^{(n)} \\
    &= \sum_{\sigma \in \mathfrak{S}_{(1^k,n-k)}/\mathfrak{S}_{(1^k,r,n-k-r)}} t^{-\ell(\sigma)} \epsilon_k^{(n)}  t^{rn}Y_{k+1}^{(n)} \cdots  Y_{k+r}^{(n)}  \epsilon_k^{(n)} \\
    &= \left(\sum_{\sigma \in \mathfrak{S}_{(1^k,n-k)}/\mathfrak{S}_{(1^k,r,n-k-r)}} t^{-\ell(\sigma)} \right) \epsilon_k^{(n)}   t^{(n-k)+\ldots + (n-k-r+1)} (T_k\cdots T_1)(T_{k+1}\cdots T_2)\cdots (T_{k+r-1}\cdots T_r) \pi_n^{r} \\
    & ~~~ \times (T_{n-r}^{-1}\cdots T_{k+1}^{-1})\cdots (T_{n-1}^{-1}\cdots T_{k+r}^{-1})  \epsilon_k^{(n)} \\
    &= \left(\sum_{\sigma \in \mathfrak{S}_{(1^k,n-k)}/\mathfrak{S}_{(1^k,r,n-k-r)}} t^{-\ell(\sigma)} \right) \epsilon_k^{(n)}   t^{(n-k)+\ldots + (n-k-r+1)} (T_k\cdots T_1)(T_{k+1}\cdots T_2)\cdots (T_{k+r-1}\cdots T_r) \pi_n^{r} \epsilon_k^{(n)} \\
    &= t^{{r \choose 2} +r}\left(\sum_{\sigma \in \mathfrak{S}_{(1^k,n-k)}/\mathfrak{S}_{(1^k,r,n-k-r)}} t^{{n-k \choose 2} - {n-k-r \choose 2} - {r \choose 2}-\ell(\sigma)} \right) \epsilon_k^{(n)}    (T_k\cdots T_1)(T_{k+1}\cdots T_2)\cdots (T_{k+r-1}\cdots T_r) \pi_n^{r} \epsilon_k^{(n)} \\
    &= t^{{r +1\choose 2}} \frac{[n-k]_{t}!}{[n-k-r]_{t}![r]_t!}\epsilon_k^{(n)}    (T_k\cdots T_1)(T_{k+1}\cdots T_2)\cdots (T_{k+r-1}\cdots T_r) \pi_n^{r} \epsilon_k^{(n)} \\
    &= t^{{r +1\choose 2}} \left(\frac{1-t^{n-k-r+1}}{1-t} \right) \cdots \left(\frac{1-t^{n-k}}{1-t^r} \right) \epsilon_k^{(n)}    (T_k\cdots T_1)(T_{k+1}\cdots T_2)\cdots (T_{k+r-1}\cdots T_r) \pi_n^{r} \epsilon_k^{(n)}.\\
\end{align*}

\end{proof}

\begin{remark}
Note that Lemma \ref{important lemma} shows that in any $\sH_{n}^{+}$ module V, the operators $e_r\left[t^nY_{k+1}^{(n)}+\ldots + t^nY_n^{(n)}\right]$ preserve the subspace $\epsilon_k^{(n)}V.$ 
\end{remark}

The next result will be important for the proof of Theorem \ref{convergence for elementary sym} where we will need to argue that the operator $e_r\left[t^nY_{k+1}^{(n)}+\ldots + t^nY_n^{(n)}\right]$ preserves the space $x_1\ldots x_k \mathbb{Q}(q,t)[x_1,\ldots, x_n]^{\mathfrak{S}_{(1^k,n-k)}}$ in the polynomial representation of $\sH_{n}^{+}.$

\begin{lem}\label{commutation with product of x's}
\begin{align*}
    &e_r\left[t^nY_{k+1}^{(n)}+\ldots + t^nY_n^{(n)}\right] \epsilon_k^{(n)}X_1\cdots X_k\\
    &= X_1\cdots X_k t^{rk+ {r +1\choose 2}} \left(\frac{1-t^{n-k-r+1}}{1-t} \right) \cdots \left(\frac{1-t^{n-k}}{1-t^r} \right)\epsilon_k^{(n)} (T_k^{-1}\cdots T_1^{-1})(T_{k+1}^{-1}\cdots T_2^{-1})\cdots (T_{k+r-1}^{-1}\cdots T_r^{-1}) \pi_n^{r} \epsilon_k^{(n)}\\
\end{align*}
\end{lem}

\begin{proof}
    \begin{align*}
        &e_r\left[t^nY_{k+1}^{(n)}+\ldots + t^nY_n^{(n)}\right] \epsilon_k^{(n)}X_1\cdots X_k\\
        &=t^{{r +1\choose 2}} \left(\frac{1-t^{n-k-r+1}}{1-t} \right) \cdots \left(\frac{1-t^{n-k}}{1-t^r} \right) \epsilon_k^{(n)}    (T_k\cdots T_1)(T_{k+1}\cdots T_2)\cdots (T_{k+r-1}\cdots T_r) \pi_n^{r} \epsilon_k^{(n)}X_1\cdots X_k \\
        &=t^{{r +1\choose 2}} \left(\frac{1-t^{n-k-r+1}}{1-t} \right) \cdots \left(\frac{1-t^{n-k}}{1-t^r} \right) \epsilon_k^{(n)}    (T_k\cdots T_1)(T_{k+1}\cdots T_2)\cdots (T_{k+r-1}\cdots T_r) \pi_n^{r}X_1\cdots X_k \epsilon_k^{(n)} \\
        &=t^{{r +1\choose 2}} \left(\frac{1-t^{n-k-r+1}}{1-t} \right) \cdots \left(\frac{1-t^{n-k}}{1-t^r} \right) \epsilon_k^{(n)}    (T_k\cdots T_1)(T_{k+1}\cdots T_2)\cdots (T_{k+r-1}\cdots T_r) X_{r+1}\cdots X_{k+r}\pi_n^{r} \epsilon_k^{(n)}.\\
    \end{align*}

    Further, 
    \begin{align*}
        &T_{k+r-1}\cdots T_r X_{r+1}\cdots X_{r+k} \\
        &= T_{k+r-1}\cdots T_{r+1}(T_rX_{r+1}) X_{r+2}\cdots X_{r+k} \\
        &= T_{k+r-1}\cdots T_{r+1}(tX_rT_r^{-1})X_{r+2}\cdots X_{r+k} \\
        &= tT_{k+r-1}\cdots T_{r+1}X_r X_{r+2}\cdots X_{r+k} T_r^{-1}\\
        &= tX_rT_{k+r-1}\cdots T_{r+1}X_{r+2}\cdots X_{r+k} T_r^{-1}\\
        &= tX_rT_{k+r-1}\cdots T_{r+2}tX_{r+1}T_{r+1}^{-1}X_{r+3}\cdots X_{r+k} T_r^{-1}\\
        &= t^2 X_rX_{r+1}T_{k+r-1}\cdots T_{r+2} X_{r+3}\cdots X_{r+k}T_{r+1}^{-1}T_r^{-1} \\
        &= \ldots \\
        &= t^k X_r\cdots X_{r+k-1}T_{k+r-1}^{-1}\cdots T_r^{-1}.\\
    \end{align*}
    By applying this argument repeatedly we find that 
    \begin{align*}
        &(T_k\cdots T_1)\cdots (T_{k+r-2}\cdots T_{r-1})(T_{k+r-1}\cdots T_r) X_{r+1}\cdots X_{k+r} \\
        &= t^k (T_k\cdots T_1)\cdots (T_{k+r-2}\cdots T_{r-1})X_{r}\cdots X_{k+r-1} (T_{k+r-1}^{-1}\cdots T_r^{-1}) \\
        &= \ldots \\
        &= t^{rk} X_1\cdots X_k (T_k^{-1}\cdots T_1^{-1})\cdots (T_{k+r-1}^{-1}\cdots T_r^{-1}) \\
    \end{align*}
    so therefore, 
    \begin{align*}
        &t^{{r +1\choose 2}} \left(\frac{1-t^{n-k-r+1}}{1-t} \right) \cdots \left(\frac{1-t^{n-k}}{1-t^r} \right) \epsilon_k^{(n)} (T_k\cdots T_1)(T_{k+1}\cdots T_2)\cdots (T_{k+r-1}\cdots T_r) X_{r+1}\cdots X_{k+r}\pi_n^{r} \epsilon_k^{(n)} \\
        &= t^{{r +1\choose 2}} \left(\frac{1-t^{n-k-r+1}}{1-t} \right) \cdots \left(\frac{1-t^{n-k}}{1-t^r} \right) \epsilon_k^{(n)} t^{rk} X_1\cdots X_k (T_k^{-1}\cdots T_1^{-1})\cdots (T_{k+r-1}^{-1}\cdots T_r^{-1})\pi_n^{r} \epsilon_k^{(n)}\\
        &= X_1\cdots X_k t^{rk+ {r +1\choose 2}} \left(\frac{1-t^{n-k-r+1}}{1-t} \right) \cdots \left(\frac{1-t^{n-k}}{1-t^r} \right) \epsilon_k^{(n)}  (T_k^{-1}\cdots T_1^{-1})\cdots (T_{k+r-1}^{-1}\cdots T_r^{-1})\pi_n^{r} \epsilon_k^{(n)}.\\
    \end{align*}
\end{proof}

Lastly, we have the standard coproduct formula for the elementary symmetric functions.

\begin{lem}\label{coproduct lemma}
    $e_r[X+Y] =  \sum_{s =0}^{r} e_{s}[X]e_{r-s}[Y].$
\end{lem}

\begin{proof}
    This result is standard.
\end{proof}

\subsection{Proof of Convergence}
We will now use all of the lemmas proven above to show the following result:
\begin{prop}\label{convergence for elementary sym}
    For $ r \geq 1$ the sequence of operators $(\Psi_{e_r}^{(n)})_{n \geq 1}$ converges to an operator $e_r[\Delta]$ on $\sP_{as}^{+}.$ The operators $e_r[\Delta]$ satisfy the following properties: 
    \begin{itemize}
        \item $e_r[\Delta](\widetilde{E}_{(\mu|\lambda)}) = e_r[\kappa_{\sort(\mu*\lambda)}(q,t)]\widetilde{E}_{(\mu|\lambda)}$
        \item $[e_r[\Delta], \y_i ] = 0$
        \item $[e_r[\Delta], T_i] = 0$
        \item $[e_r[\Delta],e_s[\Delta]] = 0$
        \item $$e_r[\Delta]|_{x_1\cdots x_k\sP(k)^{+}} = \sum_{s = 0}^{r} \left[ \prod_{i=1}^{s}\left( \frac{t^i}{1-t^i}\right) \right] e_{r-s}(\y_1, \ldots, \y_k) \epsilon_k (T_k\cdots T_1)(T_{k+1}\cdots T_2)\cdots (T_{k+s-1}\cdots T_s) \pi^s.$$
    \end{itemize}
\end{prop}

\begin{proof}
    The structure of the following argument is similar to the proof of Theorem 59 in \cite{MBWArxiv}. 

    Notice that every element of $\sP_{as}^{+}$ is a finite $\mathbb{Q}(q,t)$-linear combination of terms of the form $T_{\sigma}x^{\lambda}F[X]$ where $\sigma$ is a permutation, $\lambda$ is a partition, and $F\in \Lambda$. Therefore, to show that the sequence of operators $(\Psi_{e_r}^{(n)})_{n\geq 1}$ converges it suffices to show that sequences of the form 
    $$(\Psi_{e_r}^{(n)}(T_{\sigma}x^{\lambda}F[x_1+ \ldots +x_n]))_{n \geq 1} $$ converge. For $n$ sufficiently large, $T_{\sigma}$ commutes with $\Psi_{e_r}^{(n)} = e_r[t^{n}Y_{1}^{(n)}+ \ldots + t^nY_{n}^{(n)}]$ so it suffices to consider only sequences of the form $$(\Psi_{e_r}^{(n)}(x^{\lambda}F[x_1+ \ldots +x_n])) _{n \geq 1} .$$ Let $\lambda$ be a partition, $k := \ell(\lambda)$, $F\in \Lambda$, and take $n > k +r$. Set $\lambda':= (\lambda_1-1,\ldots, \lambda_k-1).$ Recall that $\widetilde{Y}_1^{(n)}X_1 = t^nY_{1}^{(n)}X_1$ from which it follows directly that $\widetilde{Y}_i^{(n)}X_i = t^nY_{i}^{(n)}X_i$ for all $1 \leq i \leq n$. Then for all $1\leq i \leq k$ we have that  
    $$t^nY_i^{(n)}X_1\cdots X_k = \widetilde{Y}_{i}^{(n)}X_1\cdots X_k.$$ This means that as operators on $x_1\ldots x_k\sP_k^{+}$, $t^nY_i^{(n)} = \widetilde{Y}_{i}^{(n)}.$ Note that these operators preserve the subspace $x_1\ldots x_k\sP_k^{+}$. Further, we may naturally extend this argument to show that as operators on $x_1\cdots x_k\sP_k^{+}$ for any $a_1,\ldots, a_k \geq 0$, and any permutation $\gamma \in \mathfrak{S}_{k}$, $$(t^nY_{\gamma(1)}^{(n)})^{a_1}\cdots (t^nY_{\gamma(k)}^{(n)})^{a_k}|_{x_1\cdots x_k\sP_k^{+}} = (\widetilde{Y}_{\gamma(1)}^{(n)})^{a_1}\cdots (\widetilde{Y}_{\gamma(k)}^{(n)})^{a_k}|_{x_1\cdots x_k\sP_k^{+}}.$$ This is notable because, unlike the Cherednik operators $Y_i^{(n)},$ the deformed Cherednik operators $\widetilde{Y}_{i}^{(n)}$ do not mutually commute.
    Therefore again as operators on $\sP_k^{+}$, for all $0 \leq s \leq r$
    $$e_{r-s}[t^nY_1^{(n)}+ \ldots +t^nY_k^{(n)}]X_1\cdots X_k = e_{r-s}(\widetilde{Y}_{1}^{(n)}, \ldots , \widetilde{Y}_{k}^{(n)})X_1\cdots X_k.$$

    Using Lemma \ref{coproduct lemma} we now find the following:

    \begin{align*}
        &e_r[t^nY_1^{(n)}+ \ldots +t^nY_n^{(n)}](x^{\lambda}F[x_1+ \ldots +x_n]) \\
        &= e_r[(t^nY_1^{(n)}+ \ldots +t^nY_k^{(n)}) + (t^nY_{k+1}^{(n)}+ \ldots +t^nY_n^{(n)})](x^{\lambda}F[x_1+ \ldots +x_n])\\
        &=  \sum_{s = 0}^{r} e_{r-s}[t^nY_{1}^{(n)}+ \ldots +t^nY_{k}^{(n)}]e_{s}[t^nY_{k+1}^{(n)}+ \ldots +t^nY_n^{(n)}](x^{\lambda}F[x_1+ \ldots +x_n]).\\
    \end{align*}

    Importantly, since $x^{\lambda}F[x_1+\ldots+x_n]$ is symmetric in $k+1,\ldots, n$, 
    $$\epsilon_{k}^{(n)}(x^{\lambda}F[x_1+\ldots+x_n]) = x^{\lambda}F[x_1+\ldots+x_n]$$ so that by using Lemmas \ref{important lemma} and \ref{commutation with product of x's} we find

    \begin{align*}
        &\sum_{s = 0}^{r} e_{r-s}[t^nY_{1}^{(n)}+ \ldots +t^nY_{k}^{(n)}]e_{s}[t^nY_{k+1}^{(n)}+ \ldots +t^nY_n^{(n)}](x^{\lambda}F[x_1+ \ldots +x_n])\\
        &= \sum_{s = 0}^{r} e_{r-s}[t^nY_{1}^{(n)}+ \ldots +t^nY_{k}^{(n)}]e_{s}[t^nY_{k+1}^{(n)}+ \ldots +t^nY_n^{(n)}]\epsilon_k^{(n)} X_1\cdots X_k (x^{\lambda'}F[x_1+ \ldots +x_n])\\
        &= \sum_{s = 0}^{r} e_{r-s}[t^nY_{1}^{(n)}+ \ldots +t^nY_{k}^{(n)}]X_1\cdots X_k t^{sk+ {s +1\choose 2}} \left(\frac{1-t^{n-k-s+1}}{1-t} \right) \cdots \left(\frac{1-t^{n-k}}{1-t^s} \right)\epsilon_k^{(n)} \\
        &~~~ \times (T_k^{-1}\cdots T_1^{-1})(T_{k+1}^{-1}\cdots T_2^{-1})\cdots (T_{k+s-1}^{-1}\cdots T_s^{-1}) \pi_n^{r} \epsilon_k^{(n)} (x^{\lambda'}F[x_1+ \ldots +x_n])\\
        &=  \sum_{s = 0}^{r} e_{r-s}(t^n\widetilde{Y}_{1}^{(n)}, \ldots, t^n\widetilde{Y}_{k}^{(n)})X_1\cdots X_k t^{sk+ {s +1\choose 2}} \left(\frac{1-t^{n-k-s+1}}{1-t} \right) \cdots \left(\frac{1-t^{n-k}}{1-t^s} \right)\epsilon_k^{(n)} \\
        &~~~ \times (T_k^{-1}\cdots T_1^{-1})(T_{k+1}^{-1}\cdots T_2^{-1})\cdots (T_{k+s-1}^{-1}\cdots T_s^{-1}) \pi_n^{r} \epsilon_k^{(n)} (x^{\lambda'}F[x_1+ \ldots +x_n])\\
        &= \sum_{s = 0}^{r} e_{r-s}(t^n\widetilde{Y}_{1}^{(n)}, \ldots, t^n\widetilde{Y}_{k}^{(n)})t^{{s +1 \choose 2}} \left(\frac{1-t^{n-k-s+1}}{1-t} \right) \cdots \left(\frac{1-t^{n-k}}{1-t^s} \right) \epsilon_k^{(n)} \\
        &~~~ \times (T_k\cdots T_1)(T_{k+1}\cdots T_2)\cdots (T_{k+s-1}\cdots T_s) \pi_n^s (x^{\lambda}F[x_1+ \ldots +x_n]).\\
    \end{align*}

    From here it is clear that 
    \begin{align*}
        & \lim_{n} \Psi_{e_r}^{(n)}(x^{\lambda}F[x_1+\ldots + x_n])\\
        &= \sum_{s = 0}^{r}\left[ \prod_{i=1}^{s}\left( \frac{t^i}{1-t^i}\right) \right] e_{r-s}(\y_1, \ldots, \y_k) \epsilon_k (T_k\cdots T_1)(T_{k+1}\cdots T_2)\cdots (T_{k+s-1}\cdots T_s) \pi^s (x^{\lambda}F[X])\\
    \end{align*}

    which is evidently an element of $\sP(k)^{+} \subset \sP_{as}^{+}.$ Therefore, the sequence of operators $(\Psi_{e_r}^{(n)})_{n \geq 1}$ converges to an operator on $\sP_{as}^{+}$ which we will call $e_r[\Delta].$
    
    We will now prove various properties of $e_r[\Delta].$ For all $1\leq i \leq k-1$ and $0 \leq s \leq r$
    
    \begin{align*}
    &\epsilon_k (T_k\cdots T_1)(T_{k+1}\cdots T_2)\cdots (T_{k+s-1}\cdots T_s) \pi^s T_i \\
    &=  \epsilon_k (T_k\cdots T_1)(T_{k+1}\cdots T_2)\cdots (T_{k+s-1}\cdots T_s) T_{i+s} \pi^s \\
    &=   \epsilon_k (T_k\cdots T_1)(T_{k+1}\cdots T_2)\cdots (T_{k+s-2}\cdots T_{s-2})T_{i+s-1} (T_{k+s-1}\cdots T_s) \pi^s\\
    &= \ldots\\
    &=   \epsilon_k T_i(T_k\cdots T_1)(T_{k+1}\cdots T_2)\cdots (T_{k+s-1}\cdots T_s) \pi^s\\
    &= T_i\epsilon_k (T_k\cdots T_1)(T_{k+1}\cdots T_2)\cdots (T_{k+s-1}\cdots T_s) \pi^s.\\
    \end{align*}

    Therefore, for any $f \in x_1\cdots x_k \sP(k)^{+}$, by expanding $f$ into a sum of terms of the form $T_{\sigma}x^{\lambda}F[X]$ where $\sigma \in \mathfrak{S}_{k}$ and $\lambda$ is a partition with $\ell(\lambda) = k$, we find that 
    $$e_r[\Delta](f) = \sum_{s = 0}^{r} \prod_{i=1}^{s}\left( \frac{t^i}{1-t^i}\right) e_{r-s}(\y_1, \ldots, \y_k) \epsilon_k (T_k\cdots T_1)(T_{k+1}\cdots T_2)\cdots (T_{k+s-1}\cdots T_s) \pi^s (f).$$

    Now let $(\mu|\lambda) \in \Phi.$ Using Corollary 47 in \cite{MBWArxiv} (see Definition \ref{stable-limit non-sym MacD function defn}) and Proposition \ref{limits proposition} we have that 
    \begin{align*}
        & e_r[\Delta](\widetilde{E}_{(\mu|\lambda)})\\
        &= \lim_{n} e_r[t^nY_1^{(n)}+\cdots + t^nY_n^{(n)}](\epsilon_{\ell(\mu)}^{(n)}(E_{\mu*\lambda*0^{n-(\ell(\mu)+\ell(\lambda))}})) \\
        &= \lim_{n} \epsilon_{\ell(\mu)}^{(n)}e_r[t^nY_1^{(n)}+\cdots + t^nY_n^{(n)}](E_{\mu*\lambda*0^{n-(\ell(\mu)+\ell(\lambda))}})\\
        &= \lim_{n} \epsilon_{\ell(\mu)}^{(n)} \left(e_r\left[ \sum_{i=1}^{n}q^{\sort(\mu*\lambda)_i}t^{i}\right]E_{\mu*\lambda*0^{n-(\ell(\mu)+\ell(\lambda))}} \right) \\
        &= \lim_{n} e_r\left[ \sum_{i=1}^{n}q^{\sort(\mu*\lambda)_i}t^{i}\right] \epsilon_{\ell(\mu)}^{(n)}(E_{\mu*\lambda*0^{n-(\ell(\mu)+\ell(\lambda))}})\\
        &= e_r[\kappa_{\sort(\mu*\lambda)}(q,t)] \widetilde{E}_{(\mu|\lambda)}.\\
    \end{align*}

    To see that $[e_r[\Delta], T_i] = 0$ we may check directly:
    \begin{align*}
        & e_r[\Delta]T_i \\
        &= \lim_{n} e_r[t^nY_1^{(n)}+ \ldots +t^nY_n^{(n)}] T_i\\
        &= \lim_{n} T_i e_r[t^nY_1^{(n)}+ \ldots +t^nY_n^{(n)}] \\
        &= T_i e_r[\Delta]. \\
    \end{align*}

    Lastly, since the $\widetilde{E}_{(\mu|\lambda)}$ are a basis of $\sP_{as}^{+}$ (Theorem \ref{weight basis theorem}) it follows that for all $i,r,s \geq 1$,
    \begin{itemize}
        \item $[e_r[\Delta], \y_i ] = 0$
        \item $[e_r[\Delta],e_s[\Delta]] = 0.$
    \end{itemize}
\end{proof}

\begin{remark}
    Note that 
    $$(T_k\cdots T_1)(T_{k+1}\cdots T_2)\cdots (T_{k+s-1}\cdots T_s) \pi^s = (T_k\cdots T_1\pi)^{s}$$ so we have that 
    $$e_r[\Delta]|_{x_1\cdots x_k\sP(k)^{+}} = \sum_{s = 0}^{r} \left[ \prod_{i=1}^{s}\left( \frac{t^i}{1-t^i}\right) \right] e_{r-s}(\y_1, \ldots, \y_k) \epsilon_k (T_k\cdots T_1 \pi)^s.$$ 
\end{remark}

As an immediate consequence we have the following result confirming the conjecture posed in \cite{MBWArxiv}.

\begin{thm}\label{convergence theorem}
    For any symmetric function $F \in \Lambda$ the sequence of operators $(\Psi_{F}^{(n)})_{n \geq 1}$ converges to an operator on $\sP_{as}^{+}$ which we may call $F[\Delta].$ These operators satisfy the following properties:
    \begin{itemize}
        \item $F[\Delta](\widetilde{E}_{(\mu|\lambda)}) = F[\kappa_{\sort(\mu*\lambda)}(q,t)]\widetilde{E}_{(\mu|\lambda)}$
        \item $[F[\Delta],\y_i] = 0$
        \item $[F[\Delta], T_i] = 0$
        \item $[F[\Delta],G[\Delta]] = 0.$
    \end{itemize}
\end{thm}

\begin{proof}
    Recall that the ring of symmetric functions $\Lambda$ is generated algebraically by the elementary symmetric polynomials $e_1,e_2,\ldots.$ For any $F \in \Lambda$ we may write $F = f(e_1,e_2,\ldots,e_r)$ so that for all $n \geq 1$ 
    $$\Psi_{F}^{(n)} = f(\Psi_{e_1}^{(n)},\ldots, \Psi_{e_r}^{(n)}).$$ By applying Propositions \ref{product of limits proposition} and \ref{convergence for elementary sym} we find that $(\Psi_{F}^{(n)})_{n \geq 1}$ converges and that
    $$F[\Delta]:= \lim_{n} \Psi_{F}^{(n)} = f(\lim_{n} \Psi_{e_1}^{(n)},\ldots, \lim_{n} \Psi_{e_r}^{(n)}) = f(e_1[\Delta],\ldots, e_r[\Delta]).$$ For $(\mu |\lambda) \in \Phi$ we see that 
    \begin{align*}
        &F[\Delta](\widetilde{E}_{(\mu|\lambda)})\\
        &= f(e_1[\Delta],\ldots, e_r[\Delta])(\widetilde{E}_{(\mu|\lambda)})\\
        &= f(e_1[\kappa_{\sort(\mu*\lambda)}(q,t)],\ldots, e_r[\kappa_{\sort(\mu*\lambda)}(q,t)])(\widetilde{E}_{(\mu|\lambda)})\\
        &= F[\kappa_{\sort(\mu*\lambda)}(q,t)]\widetilde{E}_{(\mu|\lambda)}.\\
    \end{align*}

    The other properties follow directly from Theorem \ref{weight basis theorem} and Proposition \ref{convergence for elementary sym}.
    
\end{proof}

\begin{example}
    The operator $e_2[\Delta]$ acts on $x_1 x_2 x_3 \sP(3)^{+}$ as 
    $$(\y_1\y_2+ \y_1\y_3+\y_2\y_3) + \frac{t}{1-t}(\y_1+\y_2+\y_3)\epsilon_3T_3T_2T_1\pi + \frac{t^3}{(1-t)(1-t^2)}\epsilon_3 T_3T_2T_1 T_4T_3T_2 \pi^2.$$ If we instead consider $e_{4}[\Delta]$ acting on $\sP(0)^{+} = \Lambda[X]$ then we get 
    $$\frac{t^{10}}{(1-t)(1-t^2)(1-t^3)(1-t^4)}\epsilon \pi^4.$$

    As an example computation we have that 
    $$p_2[\Delta](\widetilde{E}_{(4,1,2|5,4,2,2,1)}) = \left(q^{10}t^2 + q^8t^4 + q^8t^6+q^4t^8+q^4t^{10}+q^4t^{12}+q^2t^{14} + q^2t^{16} + t^{18} + t^{20}+\ldots\right)\widetilde{E}_{(4,1,2|5,4,2,2,1)}.$$
\end{example}

\begin{remark}\label{deltas commute with lowering}
    It is easy to see that for all $F \in \Lambda$, 
    $\epsilon_{k}F[\Delta] = F[\Delta]\epsilon_{k}$ by direct calculation:
\begin{align*}
        &\epsilon_{k}F[\Delta]\\
        &= \lim_{n}\epsilon_{k}^{(n)}\Psi_{F}^{(n)}\\
        &= \lim_{n} \frac{1}{[n-k]_{t}!} \sum_{\sigma \in \mathfrak{S}_{(1^k,n-k)}} t^{{n-k \choose 2} - \ell(\sigma)} T_{\sigma} F[t^nY_1^{(n)}+\ldots + t^nY_{n}^{(n)}]    \\
        &= \lim_{n} \frac{1}{[n-k]_{t}!} \sum_{\sigma \in \mathfrak{S}_{(1^k,n-k)}} t^{{n-k \choose 2} - \ell(\sigma)} F[t^nY_1^{(n)}+\ldots + t^nY_{n}^{(n)}] T_{\sigma}  \\
        &= \lim_{n} F[t^nY_1^{(n)}+\ldots + t^nY_{n}^{(n)}] \frac{1}{[n-k]_{t}!} \sum_{\sigma \in \mathfrak{S}_{(1^k,n-k)}} t^{{n-k \choose 2} - \ell(\sigma)}  T_{\sigma} \\
        &= F[\Delta] \epsilon_{k}.\\
 \end{align*}
\end{remark}

\subsection{Commutation Relations}

In this section we compute some of the commutation relations between the $\Delta$-operators $F[\Delta]$ and the operator $\widetilde{\pi}$ on $\sP_{as}^{+}$. These relations are conceptually important because in the case of the finite rank DAHA in type $\GL_n$ the analogous commutation relations allow for one to develop a theory of intertwiners and the Knop-Sahi relations for non-symmetric Macdonald polynomials. 

We start with the following result which will follow easily using the properties of Ion-Wu limits and particularly Proposition \ref{product of limits proposition}.

\begin{prop}\label{commutation relation for Delta}
    For $F \in \Lambda$
    $$\widetilde{\pi}F[\Delta] =  F[\Delta + (q^{-1}-1)\y_1] \widetilde{\pi}.$$
\end{prop}

\begin{proof}
Let $F \in \Lambda$ and $G_{i}, H_{i} \in \Lambda$ such that 
$$F[X+Y] = \sum_{i} G_i[X] H_{i}[Y].$$
 We may compute directly:

 \begin{align*}
        &\widetilde{\pi} F[\Delta] \\
        &= \lim_{n} \widetilde{\pi}_{n}F[t^nY_1^{(n)}+\ldots + t^nY_{n}^{(n)}]\\
        &= \lim_{n} F[t^nY_2^{(n)} + \ldots + t^n Y_n^{(n)} + q^{-1}t^n Y_1^{(n)}] \widetilde{\pi}_{n}\\
        &= \lim_{n} F[t^nY_1^{(n)} + \ldots + t^n Y_n^{(n)} + (q^{-1}-1)t^n Y_1^{(n)}] \widetilde{\pi}_{n}\\
        &= \lim_{n} \sum_{i} G_i[t^nY_1^{(n)} + \ldots + t^n Y_n^{(n)}] H_{i}[(q^{-1}-1)t^n Y_1^{(n)}]\widetilde{\pi}_{n}\\
        &= \lim_{n} \sum_{i} G_i[t^nY_1^{(n)} + \ldots + t^n Y_n^{(n)}] H_{i}[(q^{-1}-1)t^n Y_1^{(n)}]X_1T_1^{-1}\cdots T_{n-1}^{-1}\\
        &= \lim_{n} \sum_{i} G_i[t^nY_1^{(n)} + \ldots + t^n Y_n^{(n)}] H_{i}[(q^{-1}-1) \widetilde{Y}_1^{(n)}]X_1T_1^{-1}\cdots T_{n-1}^{-1}\\
        &= \lim_{n} \sum_{i} G_i[t^nY_1^{(n)} + \ldots + t^n Y_n^{(n)}] H_{i}[(q^{-1}-1) \widetilde{Y}_1^{(n)}]\widetilde{\pi}_{n}\\
        &= \sum_{i} \left( \lim_{n} G_i[t^nY_1^{(n)} + \ldots + t^n Y_n^{(n)}] \right) \left( \lim_{n} H_{i}[(q^{-1}-1) \widetilde{Y}_1^{(n)}] \right) \left( \lim_{n} \widetilde{\pi}_{n} \right) \\
        &= \sum_{i} G_i[\Delta] H_{i}[(q^{-1}-1)\y_1] \widetilde{\pi} \\
        &= F[\Delta + (q^{-1}-1)\y_1] \widetilde{\pi}.\\
 \end{align*}
\end{proof}

As an immediate consequence we find the following. 

\begin{cor}\label{commutation relation cor}
    For every $r \geq 1$
    $$[\widetilde{\pi}, p_r[\Delta]] = (q^{-r}-1)\y_1^{r}\widetilde{\pi}.$$
\end{cor}

\begin{proof}
    Using Proposition \ref{commutation relation for Delta} applied to $F[X] = p_r[X]$ gives 
    $$\widetilde{\pi}p_r[\Delta] = p_r[\Delta + (q^{-1}-1)\y_1] \widetilde{\pi} = (p_r[\Delta] + (q^{-r}-1)\y_1^{r})\widetilde{\pi}.$$
\end{proof}

\begin{remark}
    More generally, we may write for any $F \in \Lambda$
    $$[\widetilde{\pi},F[\Delta]] = \left(F[\Delta + (q^{-1}-1)\y_1] - F[\Delta] \right)\widetilde{\pi}.$$
\end{remark}

Lastly, we compute the full commutation relations between the limit Cherednik operators $\y_i$ and $\widetilde{\pi}.$ Interestingly, most of these relations mimic the standard finite rank DAHA situation except for $\y_1\widetilde{\pi}$ which now involves $\Delta.$

\begin{prop}
    $$\y_i\widetilde{\pi} =  \begin{cases}
    \widetilde{\pi}\y_{i-1} & i > 1 \\
    \frac{\widetilde{\pi} \Delta- \Delta\widetilde{\pi}}{q^{-1}-1} & i = 1.\\
     \end{cases}$$
\end{prop}

\begin{proof}
For $i = 1$ we have that from Proposition \ref{commutation relation for Delta} 
    $$\widetilde{\pi}\Delta - \Delta \widetilde{\pi} = (q^{-1}-1)\y_1\widetilde{\pi}$$ and hence 
    $$\y_1 \widetilde{\pi} = \frac{\widetilde{\pi} \Delta- \Delta\widetilde{\pi}}{q^{-1}-1}.$$
    Let $i > 1.$ We see that 
    \begin{align*}
        &\y_i\widetilde{\pi} \\
        &= \lim_{n} \widetilde{Y}_{i}^{(n)}\widetilde{\pi}_n \\
        &= \lim_{n} \left(t^{n-i+1}T_{i-1}\cdots T_1 \rho \pi_n T_{n-1}^{-1}\cdots T_i^{-1}\right) \left( X_1 T_1^{-1}\cdots T_{n-1}^{-1} \right) \\
        &= \lim_{n} t^{n-i+1}T_{i-1}\cdots T_1 \rho \pi_n X_1 T_{n-1}^{-1}\cdots T_i^{-1} T_1^{-1}\cdots T_{n-1}^{-1}\\
        &= \lim_{n} t^{n-i+1}T_{i-1}\cdots T_1 \rho X_2 \pi_n  T_{n-1}^{-1}\cdots T_i^{-1} T_1^{-1}\cdots T_{n-1}^{-1}\\
        &= \lim_{n} t^{n-i+1}T_{i-1}\cdots T_1 X_2\rho  \pi_n  T_{n-1}^{-1}\cdots T_i^{-1} T_1^{-1}\cdots T_{n-1}^{-1}\\
        &= \lim_{n} t^{n-i+1}T_{i-1}\cdots T_2 (tX_1T_1^{-1})\rho  \pi_n  T_{n-1}^{-1}\cdots T_i^{-1} T_1^{-1}\cdots T_{n-1}^{-1} \\
        &= \lim_{n} t^{n-i+2} X_1 T_{i-1}\cdots T_2T_1^{-1} \rho \pi_n T_1^{-1}\cdots T_{i-2}^{-1} T_{n-1}^{-1}\cdots T_i^{-1}T_{i-1}^{-1}\cdots T_{n-1}^{-1}\\ 
        &= \lim_{n} t^{n-i+2} X_1 T_{i-1}\cdots T_2T_1^{-1} \rho T_{2}^{-1}\cdots T_{i-1}^{-1} \pi_n  T_{n-1}^{-1}\cdots T_i^{-1}T_{i-1}^{-1}\cdots T_{n-1}^{-1}\\
        &= \lim_{n} t^{n-i+2} X_1 T_{i-1}\cdots T_2T_1^{-1}T_{2}^{-1}\cdots T_{i-1}^{-1} \rho  \pi_n  T_{n-1}^{-1}\cdots T_i^{-1}T_{i-1}^{-1}\cdots T_{n-1}^{-1}\\
        &= \lim_{n} t^{n-i+2} X_1 T_{1}^{-1}\cdots T_{i-1}^{-1}T_{i-2}\cdots T_{1} \rho  \pi_n  T_{i-1}^{-1}\cdots T_{n-2}^{-1}T_{n-1}^{-1}\cdots T_{i-1}^{-1}\\
        &= \lim_{n} t^{n-i+2} X_1 T_{1}^{-1}\cdots T_{i-1}^{-1}T_{i-2}\cdots T_{1} \rho  T_{i}^{-1}\cdots T_{n-1}^{-1}\pi_n  T_{n-1}^{-1}\cdots T_{i-1}^{-1}\\
        &= \lim_{n} t^{n-i+2} X_1 T_{1}^{-1}\cdots T_{i-1}^{-1}T_{i-2}\cdots T_{1}T_{i}^{-1}\cdots T_{n-1}^{-1} \rho  \pi_n  T_{n-1}^{-1}\cdots T_{i-1}^{-1}\\
        &= \lim_{n} t^{n-i+2} X_1 T_{1}^{-1}\cdots T_{i-1}^{-1}T_{i}^{-1}\cdots T_{n-1}^{-1}T_{i-2}\cdots T_{1} \rho  \pi_n  T_{n-1}^{-1}\cdots T_{i-1}^{-1}\\
        &= \lim_{n}  \left(X_1 T_{1}^{-1}\cdots T_{n-1}^{-1}\right)\left(t^{n-i+2}T_{i-2}\cdots T_{1} \rho  \pi_n  T_{n-1}^{-1}\cdots T_{i-1}^{-1}\right)\\
        &= \lim_{n} \widetilde{\pi}_n\widetilde{Y}_{i-1}^{(n)}\\
        &= \widetilde{\pi}\y_{i-1}.\\
    \end{align*}

\end{proof}

\subsection{$\mathbb{B}_{q,t}^{\text{ext}}$ Action}

We refer the reader to \cite{GCM_2017} and \cite{gonzález2023calibrated} for relevant background surrounding the algebra $\mathbb{B}_{q,t}.$ Importantly, in their conventions $q,t$ have swapped roles. We will show in this short section how to relate the $\Delta$-operators $F[\Delta]$ constructed in this paper to those defined by Gonz\'{a}lez-Gorsky-Simental.

Let us begin by recalling the definition of the extended algebra $\mathbb{B}_{q,t}^{\text{ext}}$.

\begin{defn}\label{ext B qt relations}
    The algebra $\mathbb{B}_{q,t}^{\text{ext}}$ is generated by $\mathbb{B}_{q,t}$ along with loops at each vertex $k \geq 0$ labelled by $\Delta_{p_m}$ for all $m \geq 1$ such that 
    \begin{itemize}
        \item $[\Delta_{p_m},\Delta_{p_{\ell}}] = 0$
        \item $[\Delta_{p_m},T_i] = [\Delta_{p_m},z_i] = [\Delta_{p_m},d_{-}] = 0$
        \item $[\Delta_{p_m}, d_{+}] = z_1^{m}d_{+}.$
    \end{itemize}
\end{defn}

We begin by noting that the commutation relation in Corollary \ref{commutation relation cor} is very similar to the defining relations for $\mathbb{B}_{q,t}^{\text{ext}}$ but not exactly identical. To better match conventions we may consider for $m \geq 1$ the normalized operators
$$\Delta_{p_m}:= \frac{1}{q^{m}-1}\left(t^{-m}p_{m}[\Delta] - \frac{1}{1-t^m}\right).$$ Note now that for any $(\mu|\lambda) \in \Phi$ and $m \geq 1$
$$\Delta_{p_m}(\widetilde{E}_{(\mu|\lambda)}) = \left( \sum_{i \geq 1}\left(\frac{q^{m\sort(\mu*\lambda)_i}-1}{q^m-1}\right)t^{m(i-1)} \right)\widetilde{E}_{(\mu|\lambda)}.$$ In particular, $\Delta_{p_m}(1) = 0$ for all $m \geq 1.$ Using Theorem \ref{convergence theorem}, Remark \ref{deltas commute with lowering}, and Corollary \ref{commutation relation cor} the usual action of $\mathbb{B}_{q,t}$ (see the author's prior paper for more details \cite{weising2024double}) along with these new $\Delta$-operators $\Delta_{p_m}$ define a representation of $\mathbb{B}_{q,t}^{\text{ext}}$ on $L_{\bullet}:= \bigoplus_{k \geq 0}x_1\cdots x_k \sP(k)^{+}$. In particular, since 
    $$[\widetilde{\pi},p_m[\Delta]] = (q^{-m}-1)\y_1^{m}\widetilde{\pi}$$ we see that 
    $$[\Delta_{p_m},\widetilde{\pi}] = \y_1^{m}\widetilde{\pi}.$$ 

From the work of Ion-Wu \cite{Ion_2022} we know there exists a $\mathbb{B}_{q,t}$ module isomorphism $L_{\bullet} \rightarrow V_{\bullet}$ given by 
$$x_1^{a_1}\cdots x_k^{a_k}F[x_{k+1}+x_{k+2}+\ldots] \rightarrow y_1^{a_1-1}\cdots y_k^{a_k-1}F\left[\frac{X}{t-1}\right]$$ where 
$$V_{\bullet}:= \bigoplus_{k \geq 0} \mathbb{Q}(q,t)[y_1,\ldots, y_k]\otimes \Lambda[X]$$ is the $\mathbb{B}_{q,t}$ module originally considered by Carlsson-Gorsky-Mellit \cite{GCM_2017}. Now we may use the following lemma: 

\begin{lem}\label{B qt poly rep generation}\cite{gonzález2023calibrated}
    The $\mathbb{B}_{q,t}$ polynomial representation $V_{\bullet}$ is generated by $1 \in V_{0}.$
\end{lem}

From this Gonz\'{a}lez-Gorsky-Simental find that the $\mathbb{B}_{q,t}^{\text{ext}}$ module structure on $V_{\bullet}$ is uniquely determined by the commutation relations for the $\Delta_{p_m}$ operators in the following way:

\begin{cor}\cite{gonzález2023calibrated}
    We can \textit{uniquely} define the operators $\Delta_{p_m}$ on the polynomial representation $V_{\bullet}$ by setting $\Delta_{p_m}(1) = 0$ and extending via Lemma \ref{B qt poly rep generation} using the relations in Definition \ref{ext B qt relations}.
\end{cor}

Putting together all of the work in this section thus far we get the following result.

\begin{cor}\label{ext B qt cor}
    The Ion-Wu map $L_{\bullet} \rightarrow V_{\bullet}$ is an isomorphism of $\mathbb{B}_{q,t}^{\text{ext}}$ modules.
\end{cor}

\begin{remark}

In particular, the $\Delta_{p_m}$ operators on $V_{\bullet}$ of Gonz\'{a}lez-Gorsky-Simental correspond through the Ion-Wu map to limits of sums of symmetric polynomials in the Cherednik operators as in Theorem \ref{convergence theorem}. It would interesting to know if the formula from Proposition \ref{convergence for elementary sym}
$$e_r[\Delta]|_{x_1\cdots x_k\sP(k)^{+}} = \sum_{s = 0}^{r} \left[ \prod_{i=1}^{s}\left( \frac{t^i}{1-t^i}\right) \right] e_{r-s}(\y_1, \ldots, \y_k) \epsilon_k (T_k\cdots T_1 \pi)^s$$ corresponds to something geometrically meaningful on the Gonz\'{a}lez-Gorsky-Simental $\mathbb{B}_{q,t}^{\text{ext}}$ module $V_{\bullet}.$

\end{remark}

\printbibliography

\end{document}